\documentclass[a4paper,10pt]{article}

\usepackage{amssymb}
\usepackage{amsmath}
\usepackage{amsthm}
\usepackage{multicol}
\usepackage{multirow}
\usepackage{graphics}  
\usepackage{epsfig}       

\numberwithin{equation}{section}

\theoremstyle{plain}
\newtheorem{theorem}{Theorem}[section]
\newtheorem{lemma}{Lemma}[section]
\newtheorem{proposition}[theorem]{Proposition}
\theoremstyle{definition}
\newtheorem{definition}{Definition}[section]
\theoremstyle{remark}
\newtheorem{remark}{Remark}[section]
\newtheorem{algorithm}{\normalfont\textsc{Algorithm}}

\newcommand{\R}{\mathbb{R}}
\newcommand{\norm}[1]{\|#1\|}

\title{A splitting proximal point method for Nash-Cournot equilibrium models involving nonconvex cost functions}
\author{Tran Dinh Quoc\footnote{Hanoi University of Science, Hanoi, Vietnam.\newline\textit{Present address}: Department of Electrical Engineering, ESAT-SCD and OPTEC, K.U.Leuven, Belgium. 
\newline\textit{Email}: quoc.trandinh@esat.kuleuven.be} \and Le Dung Muu\footnote{Institute of Mathematics, Hanoi, Vietnam.\newline\textit{Email}: ldmuu@math.ac.vn}}

\begin{document}

\maketitle

\begin{abstract}
Unlike convex case, a local equilibrium point of a nonconvex Nash-Cournot oligopolistic equilibrium problem may not be a global one. Finding such a \textit{local equilibrium point} or even a
\textit{stationary point} of this
problem is not an easy task. This paper deals with a numerical method for Nash-Cournot equilibrium models involving nonconvex cost functions.
We develop a local method to compute a stationary point of this class of problems. 
The convergence of the algorithm is proved and its complexity is estimated under certain assumptions.
Numerical examples are implemented to illustrate the convergence behavior of the proposed algorithm. 

\end{abstract}

\section{Introduction}\label{sec:intro}
Nash-Cournot oligopolistic equilibrium  models have been widely applied in economics, electricity markets, transportation, networks as well as in environments. Such a model can be formulated as a game
model, where each player has  a profit function which can be expressed as  the difference of its  price and cost functions. In classical models, the price function is affine while the cost
function is assumed to be convex. In this case, a local  equilibrium point is also a global one. Mathematical programming and variational inequality approaches can be used to treat this problem (see,
e.g., \cite{Contreras2004, Facchinei2003, Harker1984, Konnov2000, Murphy1982}).
  
In practical models, the cost per a unit usually decreases as the production level increases. This situation requires a nonconvex function to represent the production cost of the model.
In \cite{Muu2008}, a global optimization algorithm has been developed to find a global equilibrium point to the Nash-Cournot oligopolistic equilibrium market  model  involving piecewise concave
cost functions. However, global algorithms only work well for the problems of moderate size, while it becomes intractable when the size of the problem increases, except for special structures
are exploited. 

In this paper, we continue the work in \cite{Muu2008} by proposing a local solution method for finding a stationary point of Nash-Cournot equilibrium models with nonconvex (not necessarily
concave) cost functions. We consider a Nash-Cournot model involving an affine price function and nonconvex smooth production cost functions. With this structure, the cost function of the model can
be decomposed as the sum of a convex quadratic function and a nonconvex smooth function. 
Then, we develop a local method for finding a stationary point of such a model. The method is called a splitting proximal point algorithm.
The main idea of this algorithm is to preserve the convexity of the problem while convexifies the nonconvex part by linearizing it around each iteration point.

Proximal point methods have been well developed in optimization as well as in nonlinear analysis. Myriad of research papers concerned to these methods were published (see, e.g.,
\cite{Lewis2008,Martinet1970,Rockafellar1976,Rockafellar1997} and the references quoted therein). 
However, all these papers only deal with a class of monotone problems.
Recently, Pennanen in \cite{Pennanen2002} extended the proximal point method to nonmonotone cases for solving variational inequality problems. Lewis \cite{Lewis2008} further generalized this
algorithm in a unified framework using the prox-regularity concept \cite{Rockafellar1997}. 
A main point of the proximal point methods is to choose the proximal parameter sequence. 
This affects to the performance of the algorithm as well as its global convergence behavior. G\"{u}ler in \cite{Guler1991} investigated the rate of global convergence of the classical proximal
point methods for convex programs. The worst case complexity bound of this method is $O(1/k)$, where $k$ the number of iterations. The author further accelerated the classical proximal point method
to get a better complexity bound for the convex programming problems, precisely, $O(1/k^2)$ \cite{Guler1992} by using the idea of Nesterov for gradient methods \cite{Nesterov2004}.
A splitting proximal point are also developed in many research papers. Such a method was applied to optimization problems by Mine and Fukushima in \cite{Mine1981} and, recently, by Nesterov in
\cite{Nesterov1996}.

This paper contributes a new local method for finding a stationary point of a Nash-Cournot equilibrium models involving the nonconvex cost functions. The algorithm called splitting proximal point
method is presented and its convergence is investigated. The global worst case complexity bound is also provided, which is $O(1/\sqrt{k})$, where $k$ is the iteration counter. To our knowledge, this
is the first estimation proposed to Nash-Cournot equilibrium model.

The rest of the paper is organized as follows. Section \ref{sec:MVIP} presents a formulation of a Nash-Cournot oligopolistic  equilibrium model involving nonconvex cost function. The problem is
reformulated as a mixed variational inequality. In Section \ref{sec:LocalEP}, we define three concepts including local and global equilibria, and stationary points of the Nash-Cournot
equilibrium model.  Section \ref{sec:GDmap} deals with a gradient mapping and its properties. The splitting algorithm is described in Section \ref{sec:LocalMethod}, where its convergence is
proved and the worst-case complexity is estimated. Two numerical examples are implemented in the last section.
   
\section{ Mixed variational inequality formulation} \label{sec:MVIP}
We consider a Nash-Cournot  oligopolistic equilibrium market models with $n$-firms producing a common homogeneous commodity in a non-cooperative fashion. The price $p$ of production depends on
the total quantity $\sigma:= \sum_{i=1}^n x_i$ of the commodity. Let  $h_i(x_i)$ denote the cost of the firm $i$ when its production level is $x_i$. Suppose that the profit of  firm $i$ is given as
\begin{equation}\label{eq:f_i_function}
f_i(x_1,\cdots, x_n) := x_ip\big(\sum_{i = 1}^n x_i\big) - h_i(x_i) ,~~i = 1,\dots, n,
\end{equation}
where $h_i$ is the cost function of the firm $i$ which is assumed to only depend on its production level.

Let $ C_i \subset \mathbb{R}$ $(i = 1,\dots, n)$ denote the strategy set of the firm $i$, which is assumed to be closed and convex. Each firm seeks on its strategy set to maximize the profit by
choosing the corresponding production level under the presumption that the production of the other firms are parametric inputs. In this context, a Nash equilibrium is a production pattern in which no
firm can increase its profit by changing its controlled variables.
Thus under this equilibrium concept, each firm determines its best response given other firms' actions. 
Mathematically, a point  $x^{*}=(x^{*}_1,\dots,x^{*}_n)^T \in  C:= C_1\times \dots\times C_n$ is
said to be a Nash equilibrium point  if
\begin{equation}\label{eq:nash_optima}
 f_i(x^{*}_1, \cdots, x^{*}_{i-1}, y_i, x^{*}_{i+1}, \cdots, x^{*}_n) \leq f_i(x^{*}_1,\cdots,x^{*}_n),~ \forall y_i\in C_i ~(i=1,\dots, n).
\end{equation}

When $h_i$ is affine, this market problem can be formulated as a special Nash equilibrium problem in the $n$-person non-cooperative game model, which in turns is a strongly monotone variational
inequality (see, e.g., \cite{Konnov2000}). 

\noindent Let us define
\begin{equation}\label{eq:psi_function}
\Psi(x, y) := - \sum_{i = 1}^nf_i(x_1, \cdots, x_{i-1}, y_i, x_{i+1}, \cdots, x_n),
\end{equation}
and
\begin{equation}\label{eq:phi_function_app}
\phi(x, y)  := \Psi(x, y) - \Psi(x, x).
\end{equation}
Then, as proven in \cite{Konnov2000}, the problem of finding an equilibrium point of this model can be reformulated as the following equilibrium problem:
\begin{equation}\label{eq:EP_problem}
\textrm{Find } ~ x^{*}\in C ~\textrm{such that:}~
\phi(x^{*}, y) \geq 0 ~\textrm{for all}  ~y \in C. \tag{EP}
\end{equation}
This generalized setting was proposed by Blum and Oettli in \cite{Blum1994} (see also in \cite{Muu2008}).

In classical Cournot models \cite{Facchinei2003, Konnov2000}, the price and the cost functions for each firm are assumed to be affine and given as follows:
\begin{align}
& p(\sigma) = \alpha_0 - \beta \sigma, \quad \alpha_0 \geq 0,\quad \beta > 0,\quad \hbox{with} \; \sigma = \sum_{i=1}^n x_i, \label{eq:price_func}\\
&h_i(x_i) = \mu_ix_i + \xi_i,\quad \mu_i > 0, \quad \xi_i \geq 0 ~~ (i = 1, \dots, n). \label{eq:lncost_func} 
\end{align} 
In this case, using \eqref{eq:f_i_function}, \eqref{eq:nash_optima}, \eqref{eq:psi_function} and \eqref{eq:phi_function_app}, it is easy to check that 
\begin{equation*}
\phi(x, y) = (\tilde{B}x + \mu - \alpha)^T(y - x) + \frac{1}{2}y^TBy - \frac{1}{2}x^TBx,
\end{equation*}
where 
\begin{align*}
&B = \begin{bmatrix} 
2\beta\! &0\!&0\!&\cdots\!&0\! \\
0\! & 2\beta\! &0 \! & \cdots\! &0\! \\
\cdots\! &\cdots\! &\cdots \! &\cdots \! &\cdots \! \\
0\! &0\! &0\! &0\! &2\beta\! 
\end{bmatrix}; ~
\tilde{B} = \begin{bmatrix} 
0\! &\beta\!  &\beta\!  &\cdots\! &\beta\! \\
\beta\! &0\! &\beta \! &\cdots\! &\beta\! \\
\cdots\! &\cdots \! &\cdots\! &\cdots\! &\cdots\! \\
\beta\! &\beta\! &\beta\! &\cdots\! & 0\!
\end{bmatrix},\\
&\alpha  = (\alpha_0, \dots, \alpha_0)^T,  ~\textrm{and} ~~ \mu = (\mu_1, \dots, \mu_n)^T.
\end{align*}
Then the problem of finding a Nash equilibrium point can be formulated as a mixed variational inequality of the form:
\begin{align}\label{eq:LMVIP}
&\textrm{Find} ~x^{*} \in C ~\textrm{such that:} \notag\\
&~~(\tilde{B}x^{*}\! + \! \mu \! - \! \alpha)^T(y \! - \! x^{*}) \! + \! \frac{1}{2}y^T B y \! - \! \frac{1}{2}(x^{*})^T B x^{*} \geq 0, ~\forall y \in C. 
\end{align}
Let $Q := B + \tilde B$. Since $\beta > 0$ and matrices $\tilde{B}$ and $B$ are symmetric, it is clear that $Q$ is symmetric and positive definite. This mixed variational inequality can be
reformulated equivalently  to the following strongly convex quadratic programming problem:
\begin{equation}\label{eq:QP_prob}
\min_{x\in U}\big\{\frac{1}{2}x^TQx + (\mu - \alpha)^Tx\big\}. \tag{QP}
\end{equation}
Hence, problem \eqref{eq:QP_prob} has a unique optimal solution, which is also the unique equilibrium point of the classical oligopolistic  equilibrium  market model.
 
The oligopolistic market equilibrium models, where the profit functions $f_i$ ($i= 1,\dots,n$) of each firm is assumed to be differentiable and convex with respect to its production level $x_i$ while
the other production levels are fixed, are studied in \cite{Facchinei2003} (see also \cite{Konnov2000}). This convex model is reformulated equivalently to a monotone variational inequality.

In practical models, the production cost function $h_i$ assumed to be affine is no longer satisfied. Since the cost per a unit of the action does decrease when the quantity of the
commodity exceeds a certain amount. Taking into account this fact, in the sequel, we consider market equilibrium models where the cost function $h$ may not be convex, whereas the price function is
affine as in \eqref{eq:price_func}.
Typically, the cost function $h$ is given as:
\begin{equation}\label{eq:h_cost_func}
h(x) := \sum_{i = 1}^nh_i(x_i),
\end{equation}
where $h_i$ $(i=1, \dots, n)$ is differentiable and nonconvex. 

Let us denote by
\begin{eqnarray}
&&F(x) := \tilde{B}x - \tilde{\alpha} \nonumber\\
&&\varphi(x) := g(x) - h(x), \nonumber
\end{eqnarray}
where $\tilde{\alpha} := \alpha-\mu$, $g(x) := \frac{1}{2}x^TBx$ and $h(x)$ defined as \eqref{eq:h_cost_func}.

Obviously, matrix $\tilde{B}$ is symmetric, using the same notation $\tilde B$, $B$ and $\alpha$ as in \eqref{eq:LMVIP}, we can formulate the nonconvex Nash-Cournot equilibrium model as a mixed
variational inequality:
\begin{align}\label{eq:ncMVIP}
&\textrm{Find} ~ x^{*}\in C ~ \textrm{such that:} \notag\\
&~~~F(x^{*})^T(y - x^{*})  +  \varphi(y) - \varphi(x^{*})  \geq 0  ~\textrm{for all}~  y\in C.  \tag{\textrm{ncMVIP}}
\end{align}
Note that mixed variational inequality problems of the form  \eqref{eq:ncMVIP}, where $\varphi$ is convex, i.e. $h$ is concave, were extensively studied in the literature (see, e.g.,
\cite{Anh2005,Facchinei2003,Fukushima1992,Konnov2000,Konnov2001,Muu2009,Noor2001,Salmon2004}).

\begin{remark}\label{re:nonconvex}
If the function $\varphi$ is convex and differentiable then problem \eqref{eq:ncMVIP} can be reformulated equivalently to a classical variational inequality problem. 
More generally, it can be converted to a generalized variational inequality problem when $\varphi$ is convex and subdifferentiable (see, e.g. \cite{Konnov2000}).
However, the mixed variational inequality \eqref{eq:ncMVIP} can not equivalently transform into a variational inequality if $\varphi$ is nonconvex.
\end{remark}
Note, if we define $\phi(x,y) := F(x)^T(y-x) + \varphi(y) - \varphi(x)$ then problem \eqref{eq:ncMVIP} coincides with a nonconvex equilibrium problem of the form \eqref{eq:EP_problem}.

\section{Local equilibria and stationary points}\label{sec:LocalEP}
Unlike the convex case, if the cost function $\varphi$ of the problem \eqref{eq:ncMVIP} is nonconvex, it may not have a global equilibria even if $C$ is compact, and $F$ and $\varphi$ are continuous.
Indeed, let us consider $C := [-1, 1]\subset \mathbb{R}$, $F(x) := x$ and $\varphi(x) = -\frac{1}{2}x^2$, which is concave, then $F(x)^T(y-x) + \varphi(y)-\varphi(x) = -\frac{1}{2}(y-x)^2$.
Therefore, problem \eqref{eq:ncMVIP} corresponding to this function has no solution.
 
For a given $x\in C$, let $\mathbb{B}(x, r)$ be an open ball centered at $x$ of radius $r > 0$ in $\mathbb{R}^n$.  
 Borrowing the concepts from classical optimization, we firstly propose a local equilibria and critical points (or stationary points) of the mixed variational inequality \eqref{eq:ncMVIP}.
 
\begin{definition}
A point $x^{*}\in C$ is called a local solution (or local equilibria) to \eqref{eq:ncMVIP} if there exists a  ball $\mathbb{B}(x^{*}, r)$   such that
\begin{equation}\label{eq:local_sol}
F(x^{*})^T(y - x^{*})  +  \varphi(y) - \varphi(x^{*})  \geq 0  ~\textrm{for all}~  y\in C\cap \mathbb{B}(x^{*}, r).
\end{equation}
If $C\subseteq \mathbb{B}(x^{*}, r)$ then $x^{*}$ is called a global solution (or global equilibria) to \eqref{eq:ncMVIP}.
\end{definition}

Let 
\begin{equation}\label{eq:q_function}
\psi(x; y) := (\tilde Bx - \tilde{\alpha})^T(y - x) + \frac{1}{2}y^TBy - h(y).
\end{equation}
We consider a function $m : C \times \mathbb{R}_{++} \to \mathbb{R}$ and a mapping 
$S : C\times \mathbb{R}_{++} \rightrightarrows 2^C$ defined as follows:
\begin{eqnarray}
&&m(x; r) := \min\left\{ \psi(x; y) ~|~ y \in C \cap \mathbb{\overline{B}}(x, r) \right\}, \label{eq:m_U_map}\\
&&S(x; r) := \textrm{arg}\!\min\left\{  \psi(x; y) ~|~ y \in C\cap \mathbb{\overline{B}}(x, r) \right\}, \label{eq:S_U_map}
\end{eqnarray}
where $ \mathbb{\overline{B}}(x, r)$ stands for the closure of the open ball  $\mathbb{B}(x, r) $.
As usual, we  refer to $m$ as a local gap  function for problem \eqref{eq:ncMVIP}. 
Obviously, if $h$ is continuous and $C$ is compact then the function $m$ as well as the mapping $S$ are well-defined. If $h$ is concave then $S$ is reduced to a single valued
mapping due to the symmetric and positive definiteness of $B$.

The following proposition gives a necessary and sufficient conditions for a point to be a local or global solution to \eqref{eq:ncMVIP}.

\begin{proposition}\label{pro:gap_function} 
The function $m$ defined by \eqref{eq:m_U_map} satisfies $m(x) \leq 0$ for all $x\in C$.
Moreover, the following statements are equivalent:
\begin{itemize}
\item[]\textrm{a)} $x^{*}$ is a local solution to  \eqref{eq:ncMVIP};
\item[]\textrm{b)} There exists $\bar{r} > 0$ and $x^{*} \in C$ such that $m(x^{*}; \bar{r}) = 0$;
\item[]\textrm{c)} There exists $\bar{r} > 0$ and $x^{*} \in C$ such that $x^{*}\in S(x^{*}; \bar{r})$.
\end{itemize}
\end{proposition}

\begin{proof}
Note that $\psi(x, x) = 0$ for all $x\in C$, and for any $x\in C$ and $\bar{r} > 0$, $x\in C\cap\mathbb{B}(x,\bar{r})$. Therefore, from the definition
\eqref{eq:m_U_map} of $m$, it is clear that, with a given $\bar{r} > 0$, $m(x; \bar{r}) = \min\left\{\psi(x, y) ~|~ y\in C\cap\mathbb{B}(x,\bar{r})\right\}\leq \psi(x,x) = 0$ for every $x\in C$. 
The equivalence bewteen b) and c) is trivial. We only prove that a) is equivalent to b).

Suppose that there exists $\bar{r} > 0$ and $x^{*}\in C$ such that $m(x^{*}, \bar{r}) = 0$. It follows from the definition of $m$ that $\psi(x^{*}, y) \geq \psi(x^{*}, x^{*}) = 0$ for all $x\in
\mathbb{\overline{B}}(x^{*},\bar{r})\cap C$. In particular, $\psi(x^{*}, y) \geq 0$ for all $y\in C\cap\mathbb{B}(x^{*}, \bar{r})$. Thus $x^{*}$ is a local equilibria of
\eqref{eq:ncMVIP}.
Conversely, if $x^{*}$ is a local equilibria of \eqref{eq:ncMVIP} then there exists a neighbourhood $\mathbb{B}(x^{*}, r)$ such that $r> 0$ and
\begin{equation*}
F(x^{*})^T(y - x^{*}) +  \varphi(y) -\varphi(x^{*}) \geq F(x^{*})^T(x^{*} - x^{*}) +  \varphi(x^{*}) -\varphi(x^{*}) = 0, ~\forall y\in C\cap\mathbb{B}(x^{*}, r).
\end{equation*}
Since $r>0$ and $\mathbb{B}(x^{*},r)\subset\mathbb{R}^n$, there exists $0 < \bar{r} \leq r$ such that $\mathbb{\overline{B}}(x^{*}, \bar{r})\subseteq \mathbb{B}(x^{*}, r)$. This relation shows that
the last inequality holds for all $y\in C\cap\mathbb{\overline{B}}(x^{*},\bar{r})$. Therefore, $m(x^{*}, \bar{r}) = 0$. 
\end{proof}

Clearly, if the conclusions of Proposition \ref{pro:gap_function} hold for a given $\bar{r} > 0$ and $C \subseteq \mathbb{B}(x^{*}, \bar{r})$ then $x^{*}$ is a global solution to
\eqref{eq:ncMVIP}. 
 
Next, let us define
\begin{equation}\label{eq:feas_dir_cone}
\mathcal{F}_C(x) := \left\{d:=t(y-x) ~| ~ y\in C, ~ t\geq 0 \right\}, 
\end{equation}
a cone of all feasible directions of $C$ starting from $x \in C$. The dual cone of $\mathcal{F}_C(x)$ is the normal cone of $C$ at $x$ which is defined as
\begin{equation}\label{eq:normal_cone}
N_C(x) := \begin{cases} \left\{ w \in \mathbb{R}^n ~|~ w^T(y-x)\geq 0, ~~ y\in C\right\}, &\text{if}~ x\in C,\\
           \emptyset, &\text{otherwise}.
          \end{cases}
\end{equation}
By Proposition \ref{pro:gap_function}, a point $x\in C$ is a local solution to \eqref{eq:ncMVIP} if and only if  it solves the following optimization problem:
\begin{equation}\label{eq:nonconv_prob}
x \in \textrm{arg}\!\min\left\{(\tilde Bx-\alpha)^T(y-x) + \frac{1}{2}y^TBy - h(y) ~|~ y\in C \cap \mathbb{\overline{B}}(x, \bar{r}) \right\},
\end{equation}
for some $\bar{r} > 0$.
Since $h$ is not necessarily concave,  finding  such a point $x$ satisfying \eqref{eq:nonconv_prob}, in general, is not an easy task.
In this paper, we concentrate in finding a stationary point rather than local equilibrium. We develop a method to find such a point for \eqref{eq:ncMVIP}.  
Borrowing the concept of \textit{stationary points} in optimization, we define a stationary point (or a \textit{critical point}) for the mixed variational inequality \eqref{eq:ncMVIP} as follows. 

\begin{definition}\label{de:stationary_point}
A point $x \in C$ is called a stationary point (or critical point) to the problem  \eqref{eq:ncMVIP}  if  
\begin{equation}\label{eq:stationary_point}
0\in  Qx - \tilde{\alpha} - \nabla h(x) + N_C(x),
\end{equation} 
where $N_C$ is defined by \eqref{eq:normal_cone} and $Q := \tilde{B}  + B$.
\end{definition}
Since $N_C$ is a cone, for any $c > 0$, the inclusion \eqref{eq:stationary_point} is equivalent to
\begin{equation}\label{eq:tmp1}
0\in  c [ (\tilde B + B)x -\tilde{\alpha} - \nabla h(x)] + N_C(x).
\end{equation}
Let 
\begin{equation}\label{eq:dir_derivative}
D\phi(x; d) := [(\tilde{B} + B)x - \tilde{\alpha} - \nabla h(x)]^Td,
\end{equation}
for any $x\in C$ and $d\in\mathcal{F}_C(x)$.
Then the condition \eqref{eq:stationary_point} is equivalent to
\begin{equation}\label{eq:first_order}
D\phi(x^{*}; d) \geq 0, ~~ \forall d\in\mathcal{F}_C(x^{*}). 
\end{equation}
Let us denote by $S^{*}$ the set of stationary points of \eqref{eq:ncMVIP}.
The following lemma shows that every local equilibria of problem \eqref{eq:ncMVIP} is its stationary point. The proof is simple and short, we present here for reading convenience.

\begin{lemma}\label{le:FONC}
Suppose that $h$ is continuously differentiable on its domain. Then, every local equilibria is a stationary point of problem \eqref{eq:ncMVIP}.
\end{lemma}

\begin{proof}
Suppose that $x^{*}$ is a local equilibria of \eqref{eq:ncMVIP}. Then there exists a neighborhood $\mathbb{B}(x^{*},r)$ of $x^{*}$ such that $F(x^{*})^T(y-x^{*}) + \varphi(y) -
\varphi(x^{*}) \geq 0$ for all $y\in\mathbb{B}(x^{*}, r)\cap C$. According to Proposition \ref{pro:gap_function}, this requirement is equivalent to $x^{*}$ is a local solution of 
\begin{equation}\label{eq:opt_prob}
\min\left\{ (\tilde{B}x^{*}-\tilde{\alpha})^T(y-x^{*}) + \frac{1}{2}y^TBy - h(y) ~|~y\in \mathbb{B}(x^{*}, \bar{r})\cap C\right\},
\end{equation}
for $0 < \bar{r} \leq r$.
Since $h$ is continuous differentiable, the function inside the brackets is also continuous differentiable. Applying the first order necessary optimality condition for the smooth optimization
problem \eqref{eq:opt_prob} (see, e.g., \cite{Nesterov2004}), we obtain:
\begin{equation*}
0 \in (\tilde{B} + B)x^{*} - \tilde{\alpha} - \nabla{h}(x^{*}) + N_{C\cap\mathbb{\overline{B}}(x^{*},\bar{r})}(x^{*}).  
\end{equation*}
However, $N_{C\cap\mathbb{\overline{B}}(x^{*},\bar{r})}(x^{*}) = N_C(x^{*}) \cap N_{\mathbb{\overline{B}}(x^{*},\bar{r})}(x^{*}) = N_C(x^{*})$.

 \eqref{eq:stationary_point}.
\end{proof}

Let $\partial \delta_C(x)$ denote the subdifferential of the indicator function $\delta_C$ of $C$ at $x$. One has $\partial\delta_C(x) = N_C(x)$. 
Since matrix $B$ is symmetric and positive definite, if we define $g_1(x) := \frac{1}{2}x^TBx + \delta_C(x)$ then $\partial g_1(x) = Bx + \partial\delta_C(x)$ and this mapping is maximal
monotone. Consequently, $T^{-1}_c := (I + c\partial g_1)^{-1}$ is well-defined and single valued, where $I$ is the identity mapping (see \cite{Rockafellar1976,Rockafellar1997}).

The following proposition provides a necessary and sufficient condition for a stationary point of \eqref{eq:ncMVIP}.

\begin{proposition}\label{pro:proximal_map}  
A necessary and  sufficient condition for a point $x\in C$ to be a stationary point to problem \eqref{eq:ncMVIP} is:
\begin{equation}\label{eq:proximal_map}
x =  \left( I + c\partial g_1\right)^{-1} \left( x -  c(\tilde B x - \tilde{\alpha}) + c\nabla h(x) \right),
\end{equation} 
where $c > 0$ and $I$ stands for the identity mapping.
\end{proposition}

\begin{proof}
Since $g_1$ is proper  closed  convex, the inverse $\big(I  + \partial g_1\big )^{-1}$  is single valued  and defined everywhere \cite{Rockafellar1976}.
Thus $x$ satisfies \eqref{eq:proximal_map} if and only if $ x - c (\tilde B x - \tilde{\alpha}) + c \nabla h(x)  \in (I + c\partial g_1) (x)$.  Moreover, since $N_C(x)$ is a cone and $\partial g_1(x) = Bx + \partial \delta_C(x) = Bx +  N_C(x)$, the latter inclusion is equivalent to $0\in   \tilde B x - \tilde{\alpha} + Bx - \nabla h(x)   + N_C(x)$,
which shows that $x$ is a stationary point of \eqref{eq:ncMVIP}.
\end{proof}

Now, if we define $y_c(x):= x - c(\tilde B x - \alpha) + c\nabla h(x) $  and 
\begin{equation}\label{eq:proximal_point}
S_c(x) := \big( I + c\partial g_1 \big )^{-1} \big ( x -  c(\tilde B x - \alpha) + c\nabla h(x) \big ),
\end{equation} 
then, it follows from Proposition \ref{pro:proximal_map} that $x = S_c(x)$. 
Therefore, every stationary point $x$ of \eqref{eq:ncMVIP} is a fixed-point of $S_c(\cdot)$.
To compute $S_c(x)$, it requires to solve the following strongly convex quadratic problem over a convex set:
\begin{equation}\label{eq:convex_prob}
\min\left\{ \frac{1}{2}y^TBy + \frac{1}{2c}\norm{y-y_c(x)}^2 ~|~ x \in C \right\},
\end{equation}
This problem has a unique solution for any $c>0$.

Finally, we introduce the following concept, which will be used in the sequel.
For a given tolerance $\varepsilon \geq 0$, a point $x^{*}\in C$ is said to be an $\varepsilon$-stationary point to \eqref{eq:ncMVIP} if    
\begin{equation}\label{eq:e_first_order}
D\phi(x^{*}; d) \geq -\varepsilon, ~ \forall d\in\mathcal{F}_C(x^{*}), ~\norm{d}=1.
\end{equation}

\section{Gradient mapping and its properties}\label{sec:GDmap} 
By substituting $y_c(x)$ into \eqref{eq:convex_prob}, after a simple rearrangement, we can write problem \eqref{eq:convex_prob} as
\begin{equation}\label{eq:tmp4}
\min\left\{ \frac{1}{2}y^TBy + [\tilde B x - \tilde{\alpha}  - \nabla h(x) ]^T(y-x) + \frac{1}{2c}\norm{y-x}^2 ~|~ y\in C \right\}.  
\end{equation}
Now, we consider the following mappings:
\begin{eqnarray}
&& m_c(x; y) \!:= \!\frac{1}{2}y^TBy \!+\! [\tilde B x \!-\! \tilde{\alpha}  \!-\! \nabla h(x)]^T\!(y-x) \!-\! h(x) \!+\! \frac{1}{2c}\norm{y-x}^2,\label{eq:m_function}\\
&&\textrm{and} ~ s_c(x) := \textrm{arg}\!\min \left\{ m_c(x;y) ~|~ y\in C \right\}. \label{eq:s_function} 
\end{eqnarray}
Then, since problem \eqref{eq:s_function} is strongly convex, $s_c(x)$ is well-defined and single-valued.
Let us define
\begin{equation}\label{eq:gradient_map}
G_c(x) := \frac{1}{c}[x - s_c(x)].
\end{equation}
The mapping $G_c(\cdot)$ is referred as a {\it gradient-type mapping} of \eqref{eq:m_U_map} \cite{Nesterov1996}. 
Appyling the optimality condition for \eqref{eq:s_function}  we have 
\begin{equation}\label{eq:opt_condition}
\left[Bs_c(x) + \tilde Bx - \tilde{\alpha} - \nabla h(x) - G_c(x)\right]^T( y - s_c(x)) \geq 0, ~~\forall y\in C.
\end{equation}

From now on, we further suppose that the cost function $h$ is Lipschitz continuous differentiable on $C$ with a Lipschitz constant $L_h > 0$, i.e.
\begin{align}
\norm{\nabla h(x)-\nabla h(y)} \leq L_h\norm{x-y}, ~\forall x, y\in C.  \label{eq:h_lipschitz}
\end{align}
By using the mean-valued theorem, it is easy to show that the condition \eqref{eq:h_lipschitz} implies
\begin{equation}\label{eq:extended_ineq}
\left| h(y) - h(x) - \nabla h(x)^T(y-x) \right| \leq \frac{1}{2}L_h\norm{y-x}^2, ~~\forall x, y\in C.
\end{equation}
The following lemma shows some properties of $D\phi(\cdot;\cdot)$.

\begin{lemma}\label{le:phi_properties}
For any $x\in C$, we have
\begin{align}
&D\phi(s_c(x); x-s_c(x)) \geq \frac{1-c(L_h+\norm{\tilde B})}{c^2}\norm{G_c(x)}^2, \label{eq:tmp_eq2} \\
&D\phi(s_c(x); y \!- \! s_c(x)) \!\geq \! -[1\! + \! c(L_h \!+ \! \norm{\tilde B})]\norm{G_c(x)}\norm{y \! - \! s_c(x)},~ \forall y\! \in\! C. \label{eq:tmp_eq3} 
\end{align}
As a consequence, for any $d\in\mathcal{F}_C(s_c(x))$ with $\norm{d}=1$, we have
\begin{equation}\label{eq:tmp_eq4}
D\phi(s_c(x); d) \geq -[1+c(L_h+\norm{\tilde B})]\norm{G_c(x)}. 
\end{equation}
\end{lemma}

\begin{proof}
From the definition of $D\phi$ in \eqref{eq:dir_derivative}, we have
\begin{eqnarray}\label{eq:le1_tmp1}
&&D\phi(s_c(x); x-s_c(x)) \!=\! \left[\tilde Bs_c(x) \! -\!  \tilde{\alpha} \! + \! Bs_c(x) \! - \! \nabla h(s_c(x)) \right]^T\left(x \! - \! s_c(x)\right) \nonumber\\
&& = \left[\tilde Bx - \tilde{\alpha} - \nabla h(x) + Bs_c(x)\right]^T\left(x-s_c(x)\right) \nonumber \\
[-1.5ex]\\[-1.5ex]
&& - \left[\tilde Bs_c(x) - \tilde{\alpha} - \nabla h(s_c(x)) - \tilde Bx - \tilde{\alpha} + \nabla h(x)\right]^T\left(s_c(x)-x\right) \nonumber \\
&& \geq \left[\tilde Bx - \tilde{\alpha} - \nabla h(x) + Bs_c(x) \right]^T\left(x-s_c(x)\right) - (L_h+\norm{\tilde B})\norm{x-s_c(x)}^2.\nonumber
\end{eqnarray}
Substituting \eqref{eq:opt_condition} into \eqref{eq:le1_tmp1} we obtain
\begin{align*}
D\phi(s_c(x); x - s_c(x)) &\geq  (\frac{1}{c} - [L_h + \norm{\tilde B})]\norm{x-s_c(x)}^2 \nonumber\\
& =  \frac{1 - c(L_h + \norm{\tilde B})}{c^2}\norm{G_c(x)}^2, \nonumber  
\end{align*}
which proves \eqref{eq:tmp_eq2}.

Using again \eqref{eq:opt_condition} and \eqref{eq:h_lipschitz} we have
\begin{align*}
&D\phi(s_c(x); y - s_c(x))  =  \left[\tilde Bs_c(x) - \tilde{\alpha} + Bs_c(x) - \nabla h(s_c(x)\right]^T(y - s_c(x)) \nonumber\\
& = \left[\tilde B s_c(x) - \tilde{\alpha} - \nabla h(s_c(x))\right]^T(y-s_c(x)) + (Bs_c(x))^T(y-x) \nonumber\\
&\geq \left[\tilde Bs_c(x) - \tilde{\alpha} - \nabla h(s_c(x))\right]^T(y-s_c(x)) \nonumber \\
& + \left[\tilde Bx - \nabla h(x) + \frac{1}{c}(s_c(x)-x)\right]^T(s_c(x)-y) \nonumber\\
& = \left[\tilde Bs_c(x) - \nabla h(s_c(x)) - (\tilde Bx -\tilde{\alpha}) + \nabla h(x) \right]^T(y-s_c(x)) + G_c(x)^T(s_c(x)-y) \nonumber\\
&\geq - (L_h+\norm{\tilde B})\norm{x-s_c(x)}\norm{y-s_c(x)} - \norm{G_c(x)}\norm{y-s_c(x)} \nonumber\\
&\geq - \left[1+c(L_h+\norm{\tilde B})\right]\norm{G_c(x)}\norm{y-s_c(x)},\nonumber
\end{align*}
which proves \eqref{eq:tmp_eq3}. 

By the convexity of $C$, there exists $t\geq 0$ such that $s_c(x) + t d \in C$, where $\norm{d}=1$. If we substitute $y := s_c(x)+t d \in C$ into \eqref{eq:tmp_eq3} then we get 
\begin{equation}\label{eq:le1_tmp4}
D\phi(x;td) \geq - t(1+cL_fh+c\norm{\tilde B})\norm{G_c(x)}.
\end{equation}
If $t=0$ then \eqref{eq:tmp_eq3} automatically  holds. If $t > 0$ then by the linearity of $D$ with respect to the second argument, we divide both sides of \eqref{eq:le1_tmp4} by $t>0$ to get
\eqref{eq:tmp_eq3}.
\end{proof}
 
\begin{remark}\label{re:G_properties}
For a fixed $x\in C$, if we define $e_c(x) := \norm{G_c(x)}$ and $r_c(x) := \norm{x-s_c(x)}$ then $e_c(x)$ decreases in $c$ and $r_c(x)$ increases in $c$.
\end{remark}
Indeed, let $q(y,c):= (\tilde Bx-\tilde{\alpha})^T(y-x) + \frac{1}{2}y^TBy - h(x) - \nabla h(x)^T(y-x) + \frac{1}{2c}\norm{y-x}^2$. Then $q$ is convex jointly  in two arguments $y$ and $c$. Thus
$\omega(c):=\min_{y\in C}q(y,c)$ is convex. It is easy to see that $\omega'(c) = -\frac{1}{2}\norm{G_c(x)}^2$ increases in $c$.  Hence, $e_c(x)$ decreases in $c$. If we replace $c$ by $1/c$ in
$q(y,c)$, this function becomes concave in $c$, then by the same argument as $\omega(\cdot)$, we conclude that $r_c(x)$ increases in $c$.

Since $B$ and $\tilde{B}$ are symmetric, we consider a potential function defined as follows:
\begin{equation}\label{eq:phi_function}
\gamma(x) :=  \frac{1}{2}x^TBx + \frac{1}{2}x^T\tilde B x - \tilde{\alpha}^Tx - h(x),
\end{equation}
Then, $\gamma$ is nonconvex but Lipschitz continuous differentiable.
We have the following statement.

\begin{lemma}\label{le:G_properties_2}
For $x, y\in C$, we have
\begin{align}
&m_c(x; s_c(x)) + x^T\tilde Bx - \alpha^Tx \leq \gamma(x)  - \frac{c}{2}\norm{G_c(x)}^2. \label{eq:est_01}
\end{align}
Moreover, if $c(L_h + \norm{\tilde{B}}) \leq 1$ then
\begin{align}\label{eq:descent}
m_c(x; s_c(x)) + \frac{1}{2}x^T\tilde Bx - \tilde{\alpha}^Tx \geq \gamma(s_c(x)).
\end{align}
\end{lemma}

\begin{proof}
It is obvious from the definition of $m_c(x; x)$ that $\phi(x) = m_c(x; x) + \frac{1}{2}x^T\tilde Bx - \tilde{\alpha}^Tx$. Since $m_c(x;\cdot)$ is strongly convex quadratic with modulus $\frac{1}{2c}$, using \eqref{eq:phi_function} we have
\begin{align*}
\gamma(x) - m_c(x;s_c(x))-\frac{1}{2}x^T\tilde Bx  + \tilde{\alpha}^Tx &= m_c(x;x) - m_c(x;s_c(x)) \notag\\
& \geq \frac{1}{2c}\norm{x-s_c(x)}^2  = \frac{c}{2}\norm{G_c(x)}^2,  
\end{align*}
which proves \eqref{eq:est_01}.

To prove \eqref{eq:descent}, from \eqref{eq:extended_ineq} and the definition of $\gamma$ we have 
\begin{align*}
\gamma(s_c(x)) - m_c(x; s_c(x)) - \frac{1}{2}x^T\tilde Bx + \tilde{\alpha}^Tx & = \frac{1}{2}\left[s_c(x)\tilde Bs_c(x) - x^T\tilde Bx - 2(\tilde Bx)^T(s_c(x)-x)\right] \\
&- h(s_c(x)) + h(x) + \nabla h(x)^T(s_c(x) - x) - \frac{1}{2c}\norm{s_c(x)-x}^2\\   
&\leq \frac{1}{2}(s_c(x)-x)^T\tilde B(s_c(x)-x) + \frac{(cL_h-1)}{2c}\norm{s_c(x)-x}^2\\
&\leq -\frac{1-c(L_h+\norm{\tilde B})}{2c}\norm{s_c(x)-x}^2.
\end{align*}
By assumption $c(L_h + \norm{\tilde{B}}) \leq 1$, we obtain \eqref{eq:descent}.
\end{proof}

If we combine the inequalities \eqref{eq:descent} and \eqref{eq:est_01} in Lemma \ref{le:G_properties_2} then:
\begin{equation}\label{eq:monotone}
\gamma(s_c(x)) \leq \gamma(x) - \frac{c}{2}\norm{G_c(x)} ^2.
\end{equation}
This inequality plays an important role in proving  the convergence of the splitting proximal point algorithm in the section.

For a given starting point $x^0\in C$, let us define the level set of $\gamma$ with respect to $C$ as
\begin{equation}\label{eq:level_set}
\mathcal{L}_{\gamma}(\gamma(x^0))  := \left\{ x \in C ~|~ \gamma(x) \leq \gamma(x^0) \right\}. 
\end{equation}
From \eqref{eq:monotone}, it is obvious that if $x^0\in \mathcal{L}_{\gamma}(\gamma(x^0))$ then $s_c(x^0) \in \mathcal{L}_{\gamma}(\gamma(x^0))$ provided that $c(L_h+\norm{\tilde B}) \leq 1$.

\section{A splitting proximal point algorithm and its convergence}\label{sec:LocalMethod}
Proposition \ref{pro:proximal_map} suggests that a proximal point method can be applied to find a stationary point of \eqref{eq:ncMVIP}.
For the implementation purpose, the proximal mapping defined by \eqref{eq:proximal_point} is extracted to the expression \eqref{eq:convex_prob}.
The {\it splitting proximal point algorithm} constructs an iterative sequence as follows:

\noindent\rule[1pt]{\textwidth}{1.0pt}{~~}
\begin{algorithm}
\vskip -0.2cm\label{alg:A1}{~}(The splitting proximal algorithm)
\end{algorithm}
\vskip -0.3cm
\noindent\rule[1pt]{\textwidth}{0.5pt}
\noindent{\bf Initialization:} Choose a positive number $c_0 > 0$. Find an initial point $x^0\in C$ and set $k:=0$.\\
\noindent{\bf Iteration $k$:} For a given $x^k$, execute the three steps below. 
\begin{itemize}
\item[]\textit{Step 1}: Evaluate $\nabla h(x^k)$ and set $y_k := x^k - c_k(\tilde Bx^k-\alpha)  + c_k \nabla h(x^k)$.
\item[]\textit{Step 2}: Compute $x^{k+1}$ by solving the following convex quadratic program over a convex set:
\begin{equation}\label{eq:proximal_method}
\min\left\{ \frac{1}{2}x^TBx + \frac{1}{2c_k}\norm{y-y^k}^2 ~|~ y\in C\right\}.
\end{equation}
\item[]\textit{Step 3}:  If $\norm{x^{k+1}-x^k}\leq\varepsilon$ for a given tolerance $\varepsilon>0$ then terminate, $x^k$ is an $\varepsilon$-stationary point of \eqref{eq:ncMVIP}. 
Otherwise, update $c_k$ and increase $k$ by $1$ and go back to Step 1.
\end{itemize}
\vskip-0.3cm
\noindent\rule[1pt]{\textwidth}{1.0pt}
In Algorithm \ref{alg:A1} we left unspecified the way to update $c_k$. 
If the Lipschitz constant $L_h$ is provided then we can choose $c_k = \frac{1}{L_{\gamma}}$ for all $k$, where $L_{\gamma}
:= L_h+\norm{\tilde{B}}$. Otherwise, a line-search procedure can be used to update $c_k$. 
The latter procedure is briefly described as follows.
First, we choose two constants $\underline{c}$ and $\bar{c}$ such that $\underline{c} > 0$ and
$\frac{1}{L_{\gamma}} \leq \bar{c} < +\infty$. Then we perform the following steps.
\begin{itemize}
\item Given a constant $\tau_c \in (0, 1)$. Choose an initial value of $c$ in $[\underline{c}, \bar{c}]$.
\item Compute $s_c(x^k)$. While the decreasing condition
\begin{equation}\label{eq:Armijo_cond}
\gamma(s_c(x^k)) \leq m_c(x^k; s_c(x^k)) + \frac{1}{2}(x^k)^T\tilde{B}x^k - \tilde{\alpha}^Tx^k.
\end{equation}
does not satisfy, increase $c$ by $c := \tau_cc$ and recompute $s_c(x^k)$. 
\item Set $c_{k+1} := c$.
\end{itemize}
Now, we define $\Delta x^k := x^{k+1}-x^k$ and
\begin{equation}\label{eq:Delta}
\delta_k := \min_{0\leq i\leq k}\frac{\norm{\Delta_i}^2}{2c_i}.
\end{equation}
The convergence of the {\it splitting proximal point algorithm} is stated as follows.

\begin{theorem}\label{th:convergence}
Suppose that the function $h$ is Lipschitz continuous differentiable on $C$ with a Lipschitz constant $L_h \geq 0$. Suppose further that for a given $x^0\in C$ the level set
$\mathcal{L}_{\gamma}(\gamma(x^0))$ is bounded (particularly, $C$ is bounded). Then the sequence $\{x^k\}_{k\geq 0}$ generated by Algorithm \ref{alg:A1} starting from $x^0$ satisfies:
\begin{equation}\label{eq:bound_est}
\delta_k \leq \frac{(\gamma(x^0) - \underline{\gamma})}{k+1}, ~ \forall k\geq 0,
\end{equation}
where $\underline{\gamma} := \displaystyle\inf_{x\in\mathcal{L}_{\gamma}(\gamma(x^0))}\gamma(x)$.
Moreover, for any $d \in\mathcal{F}_C(x^{i_k})$ with $\norm{d}=1$, we have
\begin{equation}\label{eq:descent_est}
D\phi(x^{i_k}; d ) \geq - (1+\underline{c}L_h+\underline{c}\norm{\tilde B})\sqrt{\frac{2(\gamma(x^0) - \underline{\gamma})}{k+1}},
\end{equation}
where $i_k$ is the index such that $c_{i_k}\norm{G_{c_{i_k}}(x^{i_k})}^2 = \Delta_k$.

As a consequence,  if the sequence $\{x^k\}$ generated by \eqref{eq:proximal_method} is bounded,  then every limit point of this sequence is a stationary point of \eqref{eq:ncMVIP}. 
The set of limit points is connected and if it is finite then the whole sequence $\{x^k\}$ converges to a stationary point of \eqref{eq:ncMVIP}.
\end{theorem}

\begin{proof}
Since $\mathcal{L}_{\gamma}(\gamma(x^0))$ is bounded by assumption, we have $\underline{\gamma} := \inf_{x\in\mathcal{L}_{\gamma}(\gamma(x^0))}\gamma(x)$ is well-defined due to the continuity of
$\gamma$ and the closedness and nonemptiness of $\mathcal{L}_{\gamma}(\gamma(x^0))$ (since $x^0\in\mathcal{L}_{\gamma}(\gamma(x^0))$.
From Step 3 of Algorithm \ref{alg:A1}, if either the constant parameter $c_k = \frac{1}{L_{\gamma}}$ or the line search procedure is used then it implies
\begin{equation}
\gamma(x^{k+1}) + \frac{1}{2c_k}\norm{x^{k+1}-x^k}^2 \leq m_{c_k}(x^{k+1})  + (\tilde{B}x^k-\tilde{\alpha})^Tx^k \leq \gamma(x^k),~~ \forall k\geq 0.
\end{equation}
Note that the whole sequence $\{x^k\}$ is contained in $\mathcal{L}_{\gamma}(\gamma(x^0))$.
Rearrange and sum up these inequalities for $k=0$ to $k=K$ we  get 
\begin{equation}\label{eq:sum_gradient}
\sum_{k=0}^K\frac{1}{2c_k}\norm{x^{k+1} - x^k}^2 \leq \gamma(x^0) - \gamma(x^{K+1}) \leq \gamma(x^0) - \underline{\gamma}.
\end{equation}
Then the inequality \eqref{eq:bound_est} directly follows from the  definition of $\delta$ in \eqref{eq:Delta}. Combining \eqref{eq:bound_est} and \eqref{eq:tmp_eq4} in Lemma
\ref{le:phi_properties} we  obtain \eqref{eq:descent_est}. 

To prove the remainder, taking into account Remark \ref{re:G_properties} and then passing to the limit as $k$ tends to $\infty$ the resulting inequality of \eqref{eq:sum_gradient}, we get 
\begin{equation*}
\sum_{k=0}^{\infty}\frac{1}{2\bar{c}}\norm{x^{k+1}-x^k}^2 < +\infty.
\end{equation*}
Since $\bar{c}<+\infty$, this inequality implies that $\lim_{k\to\infty}\norm{x^{k+1}-x^k} = 0$. Therefore, the set of limit points is connected. 
Combine this relation and the assumptions of boundedness of $\{x^k\}$ it is easy to show that every limit point of $\{x^k\}$ is a stationary point of \eqref{eq:ncMVIP}.  When the set of the limit
points is finite, the last statement of the theorem is proved similarly using the same technique as in \cite{Ostrowski1966}[Chapt. 28].
\end{proof}

\begin{remark}\label{re:complexity}
For a given tolerance $\varepsilon>0$,  according to Theorem \ref{th:convergence}, the number of iterations $k$ to get an $\varepsilon$-stationary point is $O(\varepsilon^2)$. Consequently, the
worst-case complexity of Algorithm \ref{alg:A1} is $O(1/\sqrt{k})$.
\end{remark}

\section{Numerical test}\label{sec:num_results}
In this section, we consider to numerical examples involving concave cost functions.
The aim of these examples is to estimate the number of iterations of Algorithm \ref{alg:A1} in a certain case compared to the worst-case complexity given in Theorem \ref{th:convergence}. In addition,
we also test the time profile of the algorithm when the size of problem increases.

The algorithm is implemented in Matlab 7.8.0 (R2009a) running on a Pentium IV PC desktop with 2.6GHz and 512Mb RAM.
We assume that the feasible set $C$ of \eqref{eq:ncMVIP} is a box in $\mathbb{R}^n$. Therefore, the convex problem \eqref{eq:convex_prob} reduces to quadratic programming. We solve this problem by
using the \texttt{quadprog} solver (a built-in Matlab solver).

\vskip 0.1cm
\noindent{\it Example 1. }
Suppose that the cost function $h_i(x_i)$ of the firm $i$ is given as $h_i(x_i) = c_i^0 + c_i \ln (1 + r_ix_i)$, where $c_i^0 \geq 0$ is the ceiling cost, $c_i > 0$ and $r_i > 0$ are given. 
The function $h$ becomes
\begin{equation}\label{eq:h_function_app}
h(x) = c^0 + \sum_{i=1}^n  c_i\ln( 1 + r_ix_i ) = c^0 + \ln \prod_{i=1}^n ( 1 + r_ix_i)^{c_i},
\end{equation}
where $c^0 = \sum_{i=1}^n c^0_i$.  
It is obvious that $h_i$ is well-defined if $x_i\geq 0$ and $h_i'(x_i) = \frac{c_ir_i}{1+r_ix_i}$, which implies that $h$ is differentiable on $C = \R^n_{+}$ and
\begin{equation}\label{eq:dh_function}
\nabla h(x) = (\frac{c_1r_1}{1+r_1x_1}, \dots, \frac{c_nr_n}{1+r_nx_n})^T.
\end{equation}  
Since $h_i''(x_i) = -c_ir_i^2/(1+r_ix_i)^2$, we have $ |h_i''(x_i)| \leq c_ir_i^2$ and $h$ is concave. Moreover, $\nabla h$ is Lipschitz continuous with the Lipschitz constant $L_h := \max\{
c_ir_i^2 ~|~ i=1,\dots, n\}$. 

In this example, we choose $ \beta = 0.1 > 0$, $ \alpha = 10$, $c_i^0 = 2$, $c_i = 1.5$ for all $i=1,\dots, n$, and $ r_i = 1 + \omega_i$, where $\omega_i$ is randomly generated in
$(0,1)$ ($i=1,\dots, n$). The strategy set of the firm $i$ is defined by $C_i := [0, 10]$ for all $i=1,\dots, n$.

We test Algorithm \ref{alg:A1} for problem \eqref{eq:ncMVIP} with the size increasing from $10$ to $1000$. The tolerance $\varepsilon$ is $10^{-3}$.
The number of iterations as well as the CPU time with respect to the size of problem is visualized in \textrm{Fig1.} and \textrm{Fig2.}, respectively. 

\begin{figure}[ht]
\begin{multicols}{2}{
\centerline{\includegraphics[angle=0,width=0.45\textwidth, height=3.2cm]{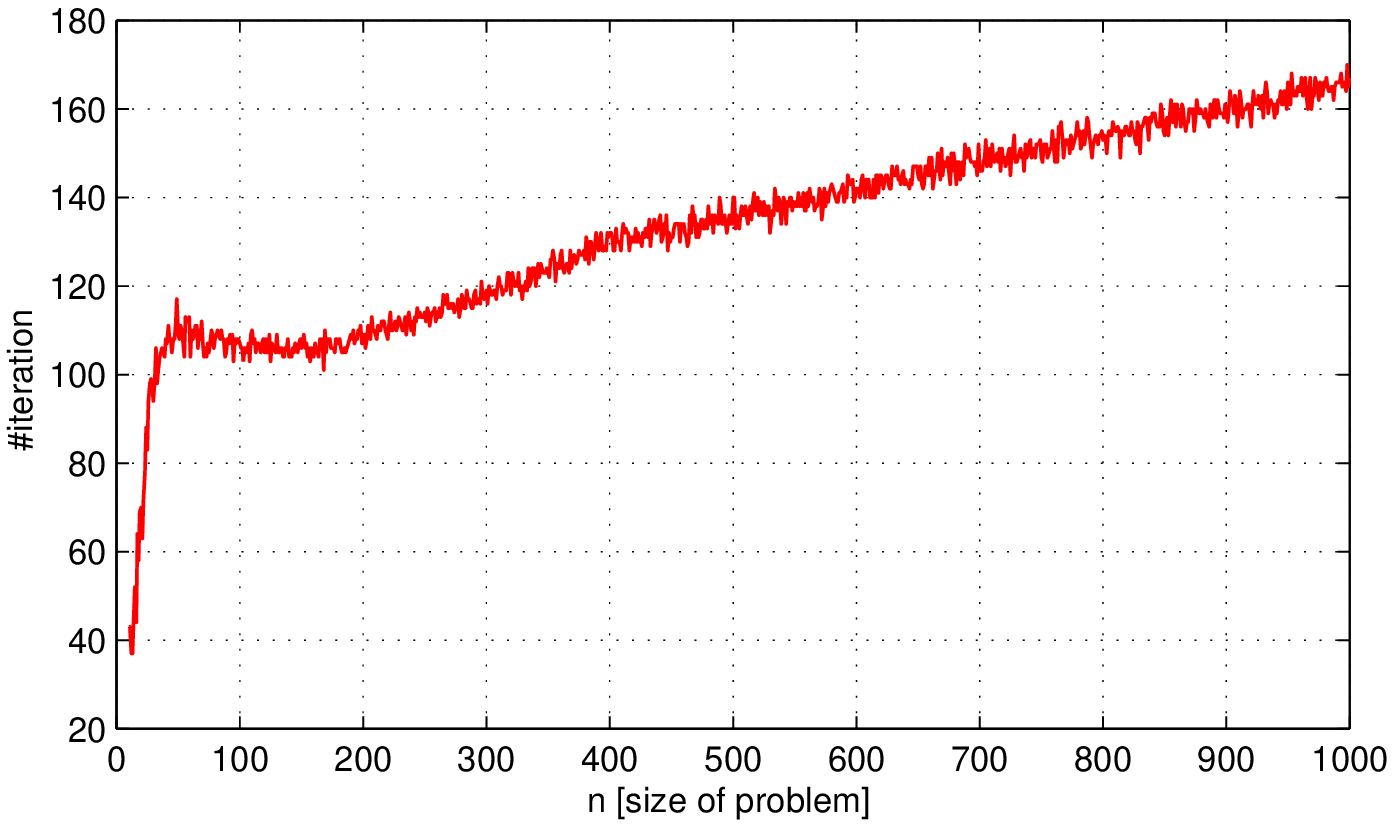}}
\label{fig:timesize1}
\centerline{\scriptsize Fig1. Number of iterations depending on $n$ [Ex. 1]}

\centerline{\includegraphics[angle=0,width=0.45\textwidth, height=3.2cm]{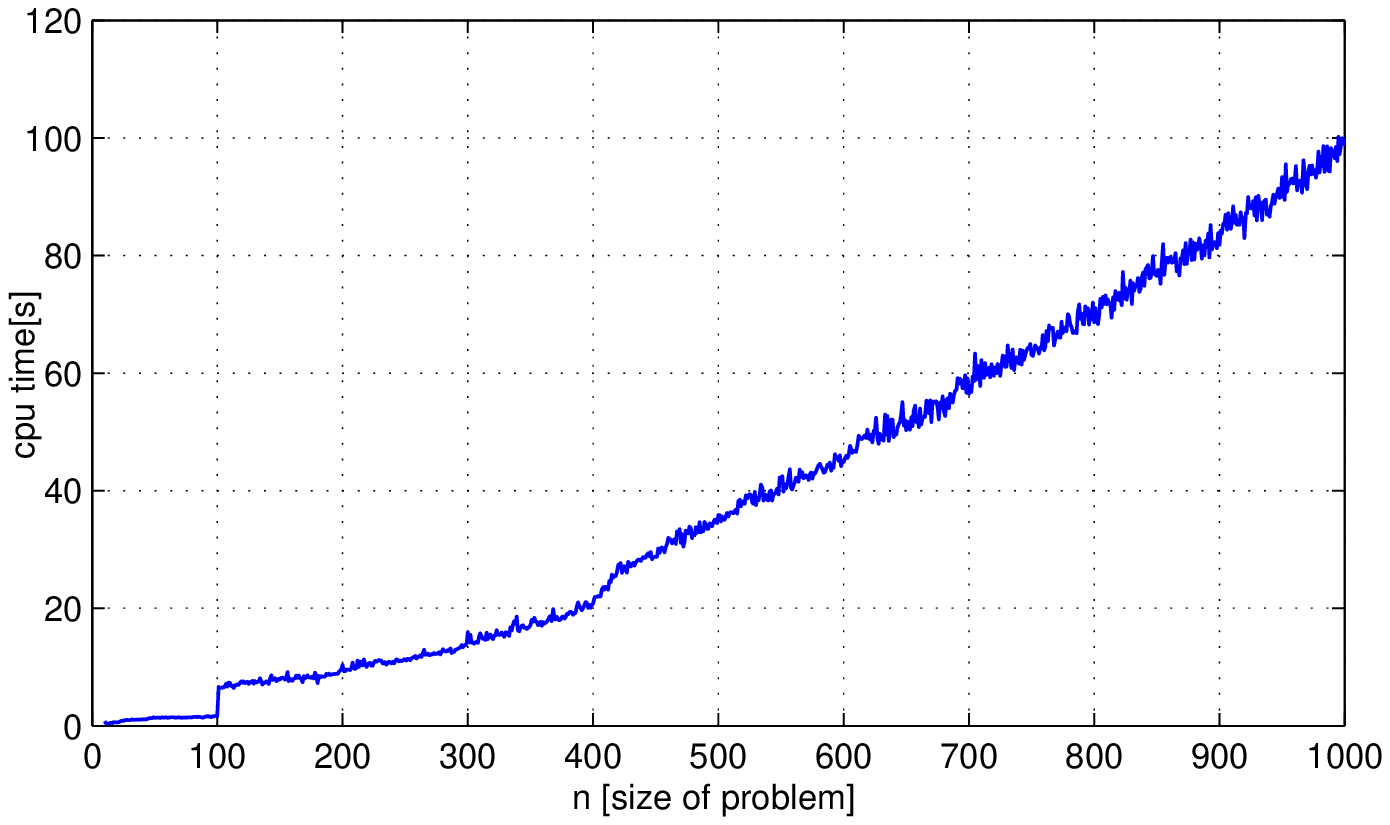}}
\label{fig:itsize1}
\centerline{\scriptsize Fig2. CPU time depending on $n$ [Ex. 1]}
}\end{multicols}
\end{figure}
From \eqref{eq:descent_est} of Theorem \ref{th:convergence}, it implies that the number of iterations $k$ to reach an $\varepsilon$-stationary point depends on the structure of the function $h$ and
the value $\gamma(x^0)-\underline{\gamma}$, $L_h$ and $\norm{\tilde B}$.
Since $h$ is a logarithm function, the value of $h$ slowly increases in $n$, while the Lipschitz constant $L_h \leq 1.5\times 2^2 = 6$ for all $n$ and the norm $\norm{\tilde B} = (n-1)\beta$.
Consequently, the worst-case complexity bound increases almost linearly in $n$. 
As can be seen from the first figure, the number of iterations increases with a small slope when the size of problem grows up. The curvature of this graph stays below a linear line generated by the
worst-case complexity bound.  The CPU time also increases almost linearly in the size of problem. 

\vskip 0.1cm
\noindent{\it Example 2. } In this example, we choose the cost function $h_i$ as $h_i(x_i) = c_i^0 -  c_i e^{-r_i x_i}$, where $c_i^0\geq c_i >0$ and $r_i>0$ given. 
It is easy to see that $h_i''(x_i) = -c_i\alpha_i^2e^{-r_ix_i} < 0$, then $h_i$ is concave. 
Since $h_i$ is differentiable on $\R$, it means that $h$ is differentiable on $\R^n$ and $\nabla h$ is expressed by
\begin{equation}\label{eq:dh_function2}
\nabla h(x) = (  c_1r_1e^{-r_1x_1}, \cdots, c_nr_ne^{-r_nx_n})^T.
\end{equation}  
We have $|h_i''(x_i)| \leq c_ir_i^2$ for all $i=1,\dots, n$, thus $\nabla h$ is Lipschitz continuous on $\R^n$ with the Lipschitz constant $L_h := \max\{c_ir_i^2 ~|~ 1\leq i\leq n\}$.

To compare with the previous example, we choose the value of the parameters $\alpha$ and $\beta$ as in Example 1. The parameters $c_i^0$ and $c_i$ are given by $c_i^0 = 4$ and $c_i = 2$ for all $i=1,
\dots, n$. The parameter $r_i := 0.1 + 0.1\textrm{rand}_i$, where $\textrm{rand}_i$ is generated randomly in $(0, 1)$.

\begin{figure}[ht]
\begin{multicols}{2}{
\centerline{\includegraphics[angle=0,width=0.45\textwidth, height=3.2cm]{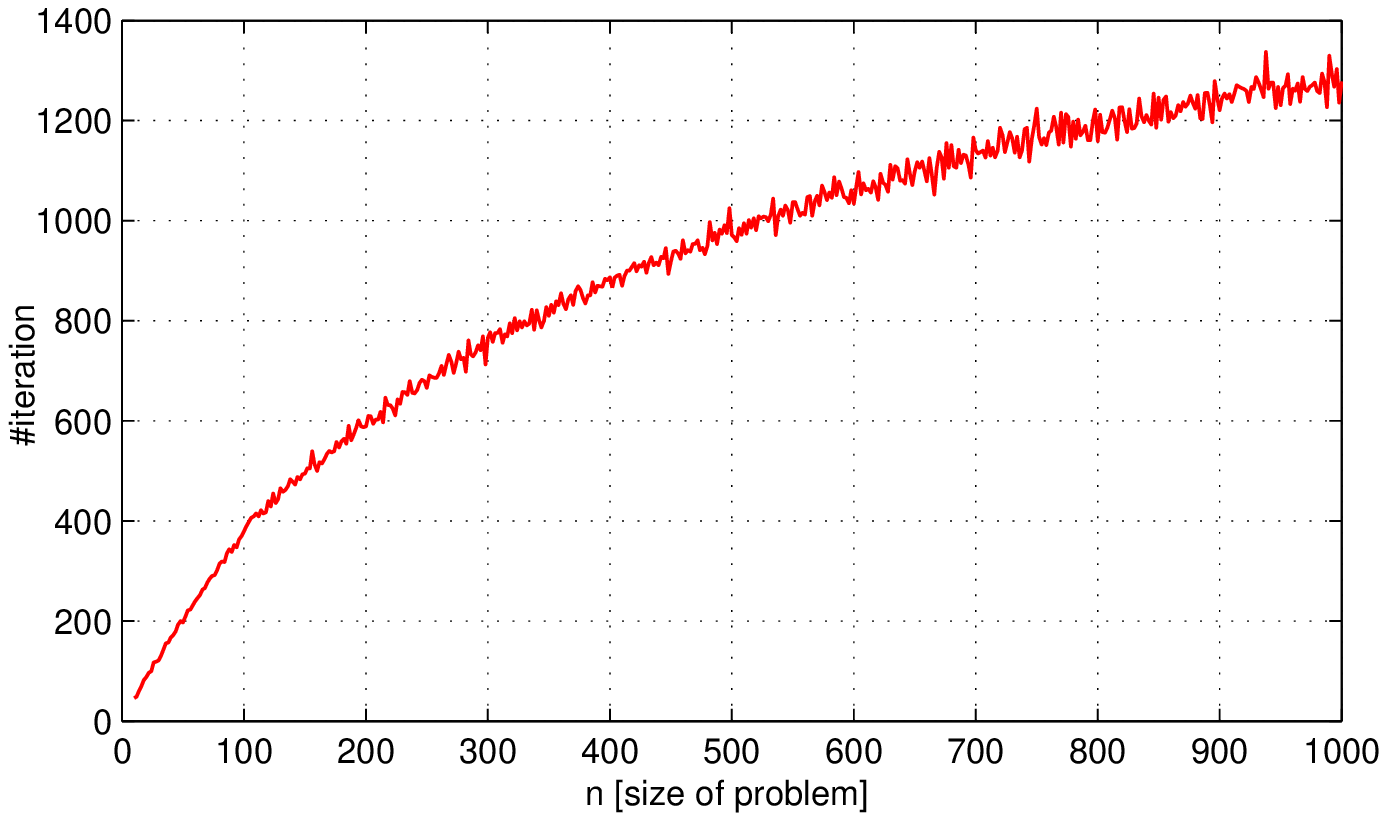}}
\label{fig:timesize2}
\centerline{\scriptsize Fig3. Iterations depending on $n$ [Ex. 2]}

\centerline{\includegraphics[angle=0,width=0.45\textwidth, height=3.2cm]{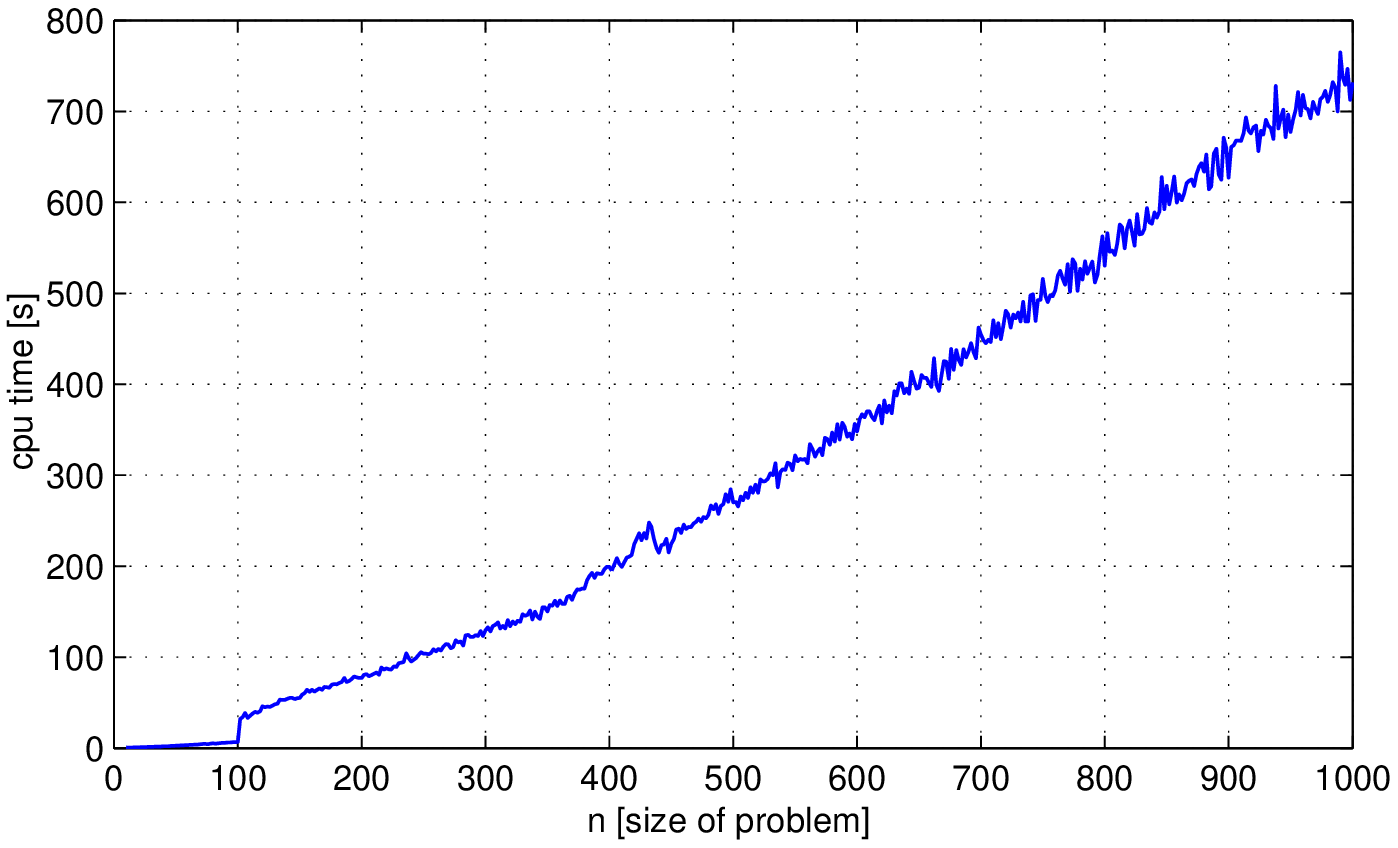}}
\label{fig:itsize2}
\centerline{\scriptsize Fig4. CPU time depending on $n$ [Ex. 2]}
}\end{multicols}
\end{figure}
We also test Algorithm \ref{alg:A1} for the problem size from $10$ to $1000$. The number of iterations and the CPU time are plotted in \textrm{Fig3.} and \textrm{Fig4}., respectively.
Since the function $\gamma$ rapidly increases in $n$ compared to the previous case, the number of iteration also increases. Consequently, the CPU time respectively increases.

\vskip 0.3cm
\noindent{\footnotesize{\textbf{Acknowledgement. }
Research supported in part by NAFOSTED, Vietnam and by Research Council KUL: CoE EF/05/006 Optimization in Engineering(OPTEC), IOF-SCORES4CHEM, GOA/10/009 (MaNet), GOA/10/11, several
PhD/postdoc and fellow grants; Flemish Government: FWO: PhD/postdoc grants, projects G.0452.04, G.0499.04, G.0211.05, G.0226.06, 
G.0321.06, G.0302.07, G.0320.08, G.0558.08, G.0557.08, G.0588.09,G.0377.09, research communities (ICCoS, ANMMM, MLDM); IWT: PhD Grants, Belgian
Federal Science Policy Office: IUAP P6/04; EU: ERNSI; FP7-HDMPC, FP7-EMBOCON, Contract Research: AMINAL. Other: Helmholtz-viCERP,
COMET-ACCM.
}}

\small{

}

\end{document}